\documentclass[12pt]{amsart}
\usepackage{amsfonts,amscd,amsthm,amsgen,amsmath,amssymb,hyperref}
\usepackage[all]{xy}
\usepackage[vcentermath]{youngtab}
\usepackage[letterpaper,hmargin=1.3in,vmargin=1.1in]{geometry}
\newtheorem{theorem}{Theorem}[section]

\newtheorem{corollary}[theorem]{Corollary}
\newtheorem{proposition}[theorem]{Proposition}

\theoremstyle{definition}
\newtheorem{definition}[theorem]{Definition}

\theoremstyle{remark}
\newtheorem{remark}[theorem]{Remark}

\numberwithin{equation}{section}

\begin{document}

\title{A Fourier-Mukai approach to the K-theory of compact Lie groups}
\author{David Baraglia}
\author{Pedram Hekmati}

\address{School of Mathematical Sciences, The University of Adelaide, Adelaide SA 5005, Australia}

\email{david.baraglia@adelaide.edu.au}
\email{pedram.hekmati@adelaide.edu.au}

\begin{abstract}
Let $G$ be a compact, connected, simply-connected Lie group. We use the Fourier-Mukai transform in twisted $K$-theory to give a new proof of the ring structure of the $K$-theory of $G$.
\end{abstract}

\thanks{This work is supported by the Australian Research Council Discovery Projects DP110103745, DP130102578 and DE12010265.}





\maketitle


\section{Introduction}

The celebrated Fourier-Mukai transform is a powerful tool employed in the study of sheaves in algebraic geometry. Moreover it has deep ties to homological mirror symmetry and the geometric Langlands program. Much less appreciated is the potential for the Fourier-Mukai transform as a $K$-theoretic tool. In this paper we will give an application of the Fourier-Mukai transform to topological $K$-theory, namely, we provide a new, conceptually simple proof of Hodgkin's theorem:

\begin{theorem}[Hodgkin \cite{hod1}]\label{thmhod}
Let $G$ be a compact, connected, simply connected semisimple Lie group of rank $n$. Then $K^*(G)$ is isomorphic to an exterior algebra over $\mathbb{Z}$ on $n$ odd generators $\rho_1, \dots , \rho_n$:
\begin{equation*}
K^j(G) = \bigoplus_{ a = j \,  ({\rm mod} \; 2) }{\bigwedge}^a_{\mathbb{Z}} \{ \rho_1 , \dots , \rho_n\}.
\end{equation*}
\end{theorem}

Theorem \ref{thmhod} can be divided into two statements. The first is that the $K$-theory of $G$ has no torsion and the second being the multiplicative structure. The proofs of both of these statements in \cite{hod1} are highly technical, raising the question of whether there are simpler arguments. A new proof of torsion-freeness was given in \cite{ara} and a simpler proof of the multiplicative structure, assuming torsion-freeness, in \cite{at1}. Theorem \ref{thmhod} can alternatively be deduced through an application of Hodgkin's equivariant K\"unneth theorem \cite{hod2}. In this approach, the hard work in proving Theorem \ref{thmhod} is shifted to the non-trivial task of establishing the equivariant K\"unneth theorem. Our proof of the theorem is independent of the equivariant K\"unneth theorem, making it arguably the shortest proof known.\\

There are three main steps to the proof, carried out in Sections \textsection \ref{secfmd}-\ref{secms}. In \textsection \ref{secfmd} we use the Fourier-Mukai transform to obtain an isomorphism between the $K$-theory of $G$ and the twisted $K$-theory of $G/T \times \hat{T}$, where $T \subset G$ is a maximal torus and $\hat{T}$ is the dual torus. In \textsection \ref{secas} we apply the Atiyah-Hirzebruch spectral sequence in twisted $K$-theory to the fibration $G/T \times \hat{T} \to \hat{T}$ in order to compute the twisted $K$-theory groups. In \textsection \ref{secms} we introduce a convolution product in twisted $K$-theory which allows us to determine the multiplicative structure of $K^*(G)$. The main theoretic tools used in the proof are twisted $K$-theory and topological T-duality. We assume familiarity with twisted $K$-theory (references \cite{bcmms},\cite{atseg1},\cite{fht} provide sufficient background), giving only a brief review of important details in \textsection \ref{seckt}. The relevant aspects of T-duality and the Fourier-Mukai transform will be reviewed where necessary.


\section{Twisted $K$-theory}\label{seckt}

There are several models that can be used to describe twists of $K$-theory. We will describe twists as bundle gerbes, following \cite{bcmms}. For a topological space $X$, a bundle gerbe $\tau$ on $X$ will be called a {\em twisting class}, or simply a {\em twist}. We let $K^*(X,\tau)$ denote the twisted $K$-theory associated to the twisting class $\tau$. We denote the tensor product of $\tau_1,\tau_2$ by $\tau_1 \otimes \tau_2$, the dual of $\tau$ by $\tau^{-1}$ and the trivial twist by $1$. The tensor product, dual and trivial twist define an abelian group structure on the set of isomorphism classes of twists, which can be naturally identified with $H^3(X,\mathbb{Z})$.\\

Recall that the group of automorphisms $1 \to 1$ of the trivial gerbe is naturally identified with $H^2(X,\mathbb{Z})$, the group of line bundles on $X$. More generally, for an isomorphism $\psi \colon \tau_1 \to \tau_2$ of twists and a line bundle $L$ on $X$, there is a naturally defined tensor product $L \otimes \psi \colon \tau_1 \to \tau_2$. This product makes the set of isomorphisms $\tau_1 \to \tau_2$ into a torsor for $H^2(X,\mathbb{Z})$, whenever $\tau_1$ and $\tau_2$ are isomorphic.\\

A trivialisation of a twist $\tau$ is defined to be an isomorphism $\lambda\colon \tau \to 1$. In terms of bundle gerbes, such a trivialisation $\lambda$ is equivalent to a rank $1$ bundle gerbe module for $\tau^{-1}$ \cite{bcmms}. Thus $\lambda$ defines a class $[\lambda] \in K^0( X , \tau^{-1})$ in twisted $K$-theory. Rank $1$ bundle gerbe modules will be referred to as {\em twisted line bundles}. The trivialisation $\lambda$ determines an isomorphism $\lambda \colon K^*(X,\tau) \to K^*(X)$, which coincides with the product $\otimes [\lambda] \colon K^*(X,\tau) \to K^*(X)$. More generally, an isomorphism $\psi \colon \tau_1 \to \tau_2$ defines a class $[\psi] \in K^0(X , \tau_2 \otimes \tau_1^{-1})$ which realises the isomorphism $\psi \colon K^*(X,\tau_1) \to K^*(X,\tau_2)$ as the product with $[\psi]$.\\

To define the Fourier-Mukai transform, we need the existence of push-forward maps in twisted $K$-theory \cite{cw0},\cite{fht}. For our purposes the following special case is sufficient. Let $f \colon X \to Y$ be a rank $n$ principal torus bundle and let $\tau$ be a twisting class on $Y$. There is a well-defined push-forward map $f_* \colon K^j(X,f^*(\tau)) \to K^{j-n}(Y,\tau)$. The two main properties of the push-forward we need are the projection formula and the base change formula. The projection formula is the identity $f_*(x) \otimes y = f_*( x \otimes f^*(y) )$, where $x \in K^*( X , f^*(\tau_1)), y \in K^*(Y , \tau_2)$. For the change of base formula, let $g \colon Z \to Y$ be any continuous map, $\tilde{f} \colon f^*(X) \to Z$ the pullback bundle and $\tilde{g} \colon f^*(X) \to X$ the naturally defined bundle map, so that $f \circ \tilde{g} = g \circ \tilde{f}$. The change of base formula is the identity $g^*( f_*(x)) = \tilde{f}_* (\tilde{g}^*(x))$, for $x \in K^*(X,f^*(\tau))$.


\section{Twisted Fourier-Mukai duality}\label{secfmd}

Recall that $G$ is a compact, connected, simply connected, semisimple Lie group of rank $n$. Let $T \subset G$ be a maximal torus in $G$. Letting $\mathfrak{t}$ denote the Lie algebra of $T$, we have $T \simeq \mathfrak{t}/\Lambda$, where $\Lambda = \pi_1(T) \simeq \mathbb{Z}^n$. Let $\hat{T}$ be the dual torus to $T$, defined as $\hat{T} = \mathfrak{t}^*/\Lambda^*$. Let $t^1 , \dots , t^n$ be a basis for $\Lambda$ and $t_1 , \dots , t_n$ the dual basis. Using $H^1( \hat{T} , \mathbb{Z} ) \simeq \Lambda$, we identify $t^1, \dots , t^n$ with a basis of $1$-forms on $\hat{T}$. Similarly $t_1 , \dots , t_n$ define a basis of $1$-forms for $T$. The projection $\pi \colon G \to G/T$ is a principal torus bundle of rank $n$ and has a Chern class $c \in H^2( G/T , \Lambda )$. Using the basis $t^1 , \dots , t^n$, we write $c = c_i t^i$, where $c_i \in H^2( G/T , \mathbb{Z} )$. This defines a twisting class $\kappa = c_i \smallsmile t^i \in H^3( G/T \times \hat{T} , \mathbb{Z})$.\\

Let $M = G/T$ and observe that $G$ and $G/T \times \hat{T}$ are torus bundles over $M$. In fact they are {\em T-dual} in the sense of \cite{bem,brs,bar0,bar}, the meaning of which we now explain. Set $X = G$, $\hat{X} = G/T \times \hat{T}$, let $q \colon \hat{X} \times_M X \to \hat{X}$ be the projection onto the first factor and $p \colon \hat{X} \times_M X \to X$ the projection to the second factor. The first requirement for $T$-duality is that the twist $\kappa$ is trivial on the fibres of $\hat{X}$, which is clearly the case here. Second, there must exist a trivialisation $\mathcal{P} \colon q^*(\kappa) \to 1$ of $q^*(\kappa)$ on $\hat{X} \times_M X$. Given a trivialisation $\tau \colon \kappa|_{\hat{T}} \to 1$ of $\kappa$ on the fibres of $\hat{X} \to M$, we may identify the restriction of $\mathcal{P}$ to the fibres of $\hat{X} \times_M X \to M$ with a line bundle $\mathcal{P}'$ on $\hat{T} \times T$, via $\mathcal{P}|_{\hat{T} \times T} = \mathcal{P}' \otimes q^*(\tau)$. We say that $\mathcal{P}$ is a {\em twisted Poincar\'e line bundle} if on each fibre of $\hat{X} \to M$ there is a trivialisation $\tau \colon \kappa|_{\hat{T}} \to 1$ for which $\mathcal{P}'$ is the Poincar\'e line bundle on $\hat{T} \times T$ (by the Poincar\'e line bundle, we mean the complex line bundle on $\hat{T} \times T$ with Chern class $t_i \smallsmile t^i$). From the existence theory for T-duals in \cite{brs,bar}, we have:

\begin{theorem}
The space $\hat{X} = G/T \times \hat{T}$ with twisting class $\kappa$ is T-dual to $X = G$ with trivial twisting class. That is, there exists a twisted Poincar\'e line bundle $\mathcal{P}$ on $\hat{X} \times_M X$.
\end{theorem}

Choose a twisted Poincar\'e line bundle $\mathcal{P}$. Being a twisted line bundle for $q^*(\kappa)^{-1}$, $\mathcal{P}$ defines a twisted $K$-theory class $\mathcal{P} \in K^*( \hat{X} \times_M X , q^*(\kappa)^{-1} )$. We use this to define the $K$-theoretic Fourier-Mukai transform $T \colon K^*(G/T \times \hat{T} , \kappa ) \to K^{*-n}(G)$ by:
\begin{equation}\label{equfmt}
T(x) = p_*( q^*(x) \otimes \mathcal{P} ).
\end{equation}

The main property of T-duality is that T-dual pairs have isomorphic twisted $K$-theories under the Fourier-Mukai transform \cite{brs,bar}. Thus:

\begin{theorem}
$T \colon K^*(G/T \times \hat{T} , \kappa ) \to K^{*-n}(G)$ is an isomorphism of abelian groups.
\end{theorem}

In the following sections we will determine the additive and multiplicative structure of $K^*(G)$ by studying the twisted $K$-theory of $G/T \times \hat{T}$.


\section{Additive structure}\label{secas}

To compute the additive structure of $K^*(G/T \times \hat{T} , \kappa )$, we apply the Atiyah-Hirzebruch spectral sequence in twisted $K$-theory to the fibration $G/T \times \hat{T} \to \hat{T}$. This gives a spectral sequence $E_r^{p,q}$ converging to $K^{p+q}(G/T \times \hat{T} , \kappa )$. Associated to the spectral sequence is a filtration:
\begin{equation}\label{equfilt}
\{ 0 \} = F^{n+1,k} \subseteq F^{n,k} \subseteq \dots \subseteq F^{1,k} \subseteq F^{0,k} = K^k(G/T \times \hat{T} , \kappa)
\end{equation}
such that the associated graded group $Gr^p( K^q(G/T \times \hat{T} , \kappa)) = F^{p,p+q}/F^{p+1,p+q}$ coincides with $E_\infty^{p,q}$. The $E_2$-page is given by:
\begin{equation*}
E_2^{p,q} = H^p( \hat{T} , \mathcal{K}^q(G/T) ),
\end{equation*}
where $\mathcal{K}^q(G/T)$ is a local system with coefficient group $K^q(G/T)$. The local system $\mathcal{K}^q(G/T)$ is the sheaf on $\hat{T}$ associated to the presheaf $\hat{T} \supseteq U \mapsto K^q( G/T \times U , \kappa|_{G/T \times U} )$. In this spectral sequence, we may consider $p$ to be integer-valued while $q$ is an integer mod $2$ (this applies also to the filtration $F^{p,q}$). Since $K^1(G/T) = 0$, we need only consider the terms $E_2^{p,0} = H^p( \hat{T} , \mathcal{K}(G/T) )$, where $\mathcal{K}(G/T) = \mathcal{K}^0(G/T)$.\\

While the fibre bundle $G/T \times \hat{T}$ is trivial, the local system $\mathcal{K}(G/T)$ has non-trivial monodromy arising from the twist $\kappa$. Observe that $\pi_1(\hat{T}) \simeq \Lambda^*$ is free abelian with generators $t_1, \dots , t_n$. Let $L_i$ be the complex line bundle on $G/T$ with Chern class $c_i$. The monodromy around the loop defined by $t_i$ is the action of the tensor product $ ( \, \cdot \, )  \otimes [L_i] \colon K(G/T) \to K(G/T)$ by the line bundle $L_i$. Let $R[T] = \mathbb{Z}[\Lambda^*] = \mathbb{Z}[t_1^{\pm} , \dots , t_n^{\pm}]$ be the group ring of $\Lambda^*$, which is also the representation ring of the torus $T$. Let $\epsilon \colon R[T] \to \mathbb{Z}$ be the augmentation defined by $\epsilon( t_i ) = 1$. This makes $\mathbb{Z}$ an $R[T]$-module. The monodromy action makes $K(G/T)$ into an $R[T]$-module, giving isomorphisms:
\begin{equation*}
H^p( \hat{T} , \mathcal{K}(G/T) ) \simeq H^p( \Lambda^* , K(G/T) ) \simeq Ext_{R[T]}^p( \mathbb{Z} , K(G/T) ).
\end{equation*}

By Poincar\'e duality $H^p( \hat{T} , \mathcal{K}(G/T) ) \simeq H_{n-p}(\hat{T} , \mathcal{K}(G/T) )$ and we have:
\begin{equation*}
H_{p}( \hat{T} , \mathcal{K}(G/T) ) \simeq H_{p}( \Lambda^* , K(G/T) ) \simeq Tor^{R[T]}_{p}( \mathbb{Z} , K(G/T) ).
\end{equation*}

Let $R[G]$ denote the representation ring of $G$. Restriction to the maximal torus gives an injection $i \colon R[G] \to R[T]$ and defines an augmentation $\epsilon_G = \epsilon \circ i \colon R[G] \to \mathbb{Z}$. This makes $\mathbb{Z}$ into an $R[G]$-module. Recall that there is an isomorphism $K(G/T) = R[T] \otimes_{R[G]} \mathbb{Z}$ of $R[T]$-modules \cite{koku} (note that the proof does not require the equivariant K\"unneth theorem). We thus have $Tor^{R[T]}_{p}( \mathbb{Z} , K(G/T) ) = Tor^{R[T]}_{p}(\mathbb{Z} , R[T] \otimes_{R[G]} \mathbb{Z} )$. We now recall the Pittie-Steinberg theorem \cite{pit,ste}, which asserts that $R[T]$ is a free $R[G]$-module. Therefore, the change of ring spectral sequence for Tor groups gives isomorphisms:
\begin{equation*}
Tor^{R[T]}_{p}(\mathbb{Z} , R[T] \otimes_{R[G]} \mathbb{Z} ) \simeq Tor^{R[G]}_p(\mathbb{Z},\mathbb{Z}).
\end{equation*}
Recall that the representation ring $R[G]$ of $G$ is a polynomial ring $R[G] = \mathbb{Z}[\sigma_1, \dots , \sigma_n]$ over the fundamental irreducible representations $\sigma_1, \dots , \sigma_n$. If we set $\tilde{\sigma}_i = \sigma_i - \epsilon_G(\sigma_i)$, then $R[G] \simeq \mathbb{Z}[\tilde{\sigma}_1 , \dots , \tilde{\sigma}_n ]$ and $\epsilon_G( \tilde{\sigma}_i ) = 0$. Let $\bigwedge^i_{\mathbb{Z}} \{ \rho_1 , \dots , \rho_n \}$ be the $i$-th exterior power over $\mathbb{Z}$ on $n$ generators $\rho_1 , \dots , \rho_n$. Recall the Koszul resolution for the $R[G]$-module $\mathbb{Z}$:
\begin{equation*}
\cdots \buildrel \partial \over \longrightarrow {\bigwedge}^1_{\mathbb{Z}} \{\rho_1 , \dots , \rho_n \} \otimes_{\mathbb{Z}} R[G] \buildrel \partial \over \longrightarrow {\bigwedge}^0_{\mathbb{Z}} \{ \rho_1 , \dots , \rho_n \} \otimes_{\mathbb{Z}} R[G] \buildrel \epsilon_G \over \longrightarrow \mathbb{Z},
\end{equation*}
where $\partial( \rho_i ) = \tilde{\sigma}_i$ \cite{macl}. Using this resolution we see that $Tor^{R[G]}_p( \mathbb{Z} , \mathbb{Z} ) \simeq \bigwedge_{\mathbb{Z}}^p \{ \rho_1 , \dots , \rho_n \}$ is a free Abelian group of rank ${n}\choose{p}$. Combining this with $K^1(G/T) = 0$, we see that
\begin{equation*}
E_2^{p,q} \simeq \left\{ \begin{matrix} \bigwedge_{\mathbb{Z}}^p \{ \rho_1 , \dots , \rho_n \} & q = 0 \; ($mod $2) \\ 0 & \, q = 1 \; ($mod $2). \end{matrix} \right.
\end{equation*}

Thus $E_2^{*,0}$ is torsion-free and has total rank $2^n$. This is the rank of $H^*(G)$, hence also the rank of $K^*(G)$. It follows that there can be no non-trivial differentials in the spectral sequence beyond this point, so that $E_2^{p,q} \simeq E_\infty^{p,q}$. Since there is no torsion there is no obstruction to splitting the filtration (\ref{equfilt}). Keeping track of even and odd degrees, we have shown that as abelian groups:
\begin{equation*}
\begin{aligned}
K^0(G) &\simeq \mathbb{Z}^{2^{n-1}}, & K^1(G) &\simeq \mathbb{Z}^{2^{n-1}}.
\end{aligned}
\end{equation*}


\section{Multiplicative structure}\label{secms}

The twisted Fourier-Mukai map $T \colon K^*(G/T \times \hat{T} , \kappa ) \to K^{*-n}(G)$ is not a ring isomorphism. In fact, the twisted $K$-theory groups of a space with non-trivial twisting class do not naturally carry a product. Instead we will show how to equip $K^*(G/T \times \hat{T} , \kappa )$ with a convolution operation, which corresponds to the product on $K^*(G)$ under the Fourier-Mukai map.\\

As $\hat{X}$ is a trivial $\hat{T}$-bundle, the group multiplication $\hat{\mu} \colon \hat{T} \times \hat{T} \to \hat{T}$ on $\hat{T}$ induces a fibrewise multiplication $\hat{m} \colon \hat{X} \times_M \hat{X} \to \hat{X}$. Since $1$-forms on $\hat{T}$ are primitive, we see that $\hat{m}^*(\kappa) \simeq \pi_1^*(\kappa) \otimes \pi_2^*(\kappa)$, where $\pi_1,\pi_2 \colon \hat{X} \times_M \hat{X} \to \hat{X}$ are the projections to the first and second factors. Let $\lambda \colon  \pi_1^*(\kappa) \otimes \pi_2^*(\kappa) \to \hat{m}^*(\kappa)$ be an isomorphism of twists.

\begin{definition}
Let $\mathcal{P}$ be a twisted Poincar\'e line bundle on $\hat{X} \times_M X$. View $\mathcal{P}$ as a trivialisation $\mathcal{P} \colon q^*(\kappa) \to 1$ of $q^*(\kappa)$. We say that $\mathcal{P}$ is multiplicative if there exists an isomorphism $\lambda \colon  \pi_1^*(\kappa) \otimes \pi_2^*(\kappa) \to \hat{m}^*(\kappa)$ for which the following diagram commutes:
\begin{equation*}
\xymatrix{
\pi_{13}^*(q^*\kappa) \otimes \pi_{23}^* (q^*\kappa) \ar[d]^{\pi_{13}^*(\mathcal{P}) \otimes \pi_{23}^*(\mathcal{P})} \ar[rr]^{\pi_{12}^*(\lambda)} & & (\hat{m} \times id)^* (q^*\kappa) \ar[d]^{(\hat{m} \times id)^*(\mathcal{P})} \\
1 \ar[rr]^{1} & & 1
}
\end{equation*}
where $\pi_{13},\pi_{23} \colon \hat{X} \times_M \hat{X} \times_M X \to \hat{X} \times_M X$ and $\pi_{12} \colon \hat{X} \times_M \hat{X} \times_M X \to \hat{X} \times_M \hat{X}$ are the projections onto the factors indicated.
\end{definition}

\begin{remark}
One may view $\mathcal{P}$ as an element $\mathcal{P} \in K^0( \hat{X} \times_M X , (q^*\kappa)^{-1} )$ and $\lambda$ as an element $\lambda \in K^0( \hat{X} \times_M \hat{X} , \hat{m}^*\kappa \otimes (\pi_1^* \kappa)^{-1} \otimes (\pi_2^* \kappa)^{-1} )$. Thus, we obtain an element:
\begin{equation}\label{equdelta}
\delta = (\hat{m} \times id)^*(\mathcal{P}) \otimes \pi_{13}^*(\mathcal{P})^{-1} \otimes \pi_{23}^*(\mathcal{P})^{-1} \otimes \pi_{12}^*(\lambda)
\end{equation}
in $K^0( \hat{X} \times_M \hat{X} \times_M X)$. In fact $\delta$ is naturally a line bundle, since it is an automorphism of the trivial twist. The twisted Poincar\'e line bundle $\mathcal{P}$ is multiplicative if and only if there is a $\lambda$ such that $\delta$ is the trivial line bundle.
\end{remark}

\begin{proposition}
There exists a multiplicative twisted Poincar\'e line bundle on $\hat{X} \times_M X$.
\end{proposition}
\begin{proof}
First note that $\hat{m}^*(\kappa) \simeq \pi_{13}^*(\kappa) \otimes \pi_{23}^*(\kappa)$, so certainly an isomorphism $\lambda \colon \hat{m}^* (\kappa) \to \pi_1^*(\kappa) \otimes \pi_2^*(\kappa)$ exists. Choose such an isomorphism. We obtain a line bundle $\delta$ on $\hat{X} \times_M \hat{X} \times_M X$ given by Equation (\ref{equdelta}). From the definition of T-duality, there exists a trivialisation $\tau \colon \kappa |_{\hat{T}} \to 1$ of $\kappa$ along the fibres of $\hat{X}$ such that on the fibres $\hat{T} \times T$ of $\hat{X} \times_M X$, the trivialisations $\mathcal{P}$ and $q^*(\tau)$ differ by the Poincar\'e line bundle $\mathcal{P}' \to \hat{T} \times T$. Since $\hat{X}$ is a trivial torus bundle over $M$, any line bundle on the fibre $\hat{T} \times \hat{T}$ extends to a line bundle on $\hat{X} \times_M \hat{X}$. Therefore we may assume that $\lambda$ restricted to the fibres $\hat{T} \times \hat{T}$ is the isomorphism induced by $\tau$, namely $( \pi_1^*(\tau) \otimes \pi_2^*(\tau) )^{-1} \circ \hat{m}^*(\tau)$. Then since the Poincar\'e line bundle is a multiplicative line bundle, it follows that $\delta$ is trivial on the fibres $\hat{T} \times \hat{T} \times T$. Now since $H^1(M,\mathbb{Z}) = 0$, we have that $\delta$ is the pullback of a line bundle on $M$. Tensoring $\mathcal{P}$ by $\delta$, we obtain a multiplicative twisted Poincar\'e line bundle.
\end{proof}

Given a multiplicative twisted Poincar\'e line bundle $\mathcal{P}$ and corresponding isomorphism $\lambda \colon  \pi_1^*(\kappa) \otimes \pi_2^*(\kappa) \to \hat{m}^*(\kappa)$, we define a convolution product $* \colon K^{i}( \hat{X} , \kappa ) \otimes K^{j}( \hat{X} , \kappa ) \to K^{i+j-n}(\hat{X} , \kappa)$ as follows. Let $x \in K^i(\hat{X} , \kappa ), y \in K^j(\hat{X} , \kappa)$. Take the external product $x \boxtimes y = \pi_1^*(x) \otimes \pi_2^*(y) \in K^{i+j}( \hat{X} \times_M \hat{X} , \pi_1^*(\kappa) \otimes \pi_2^*(\kappa) )$ and set:
\begin{equation*}
x * y = \hat{m}_* ( (x \boxtimes y) \otimes \lambda ).
\end{equation*}

Recall the Fourier-Mukai transform $T \colon K^*(G/T \times \hat{T} , \kappa ) \to K^{*-n}(G)$, which we have defined by Equation (\ref{equfmt}).

\begin{proposition}
We have $T(x * y) = T(x) \otimes T(y)$.
\end{proposition}
\begin{proof}
The proof is a direct calculation which closely parallels the corresponding result in algebraic geometry \cite{huy}:
\begin{equation*}
\begin{aligned}
T(x * y ) &= p_*( q^*(x * y) \otimes \mathcal{P} ) \\
&= p_*( q^* \hat{m}_*( ( x \boxtimes y ) \otimes \lambda ) \otimes \mathcal{P} ) \\
&= p_*( (\hat{m} \times id)_* \pi_{12}^* ( ( x \boxtimes y ) \otimes \lambda ) \otimes \mathcal{P} ) \\
&=p_*( (\hat{m} \times id)_* ( \pi_{12}^*(\pi_1^*(x) \otimes \pi_2^*(y) ) \otimes \pi_{12}^*(\lambda) \otimes (\hat{m} \times id)^*(\mathcal{P}) ) ) \\
&=p_*( (\hat{m} \times id)_* ( \pi_{12}^*(\pi_1^*(x) \otimes \pi_2^*(y) ) \otimes \pi_{12}^*(\lambda)  \otimes \pi_{12}^*(\lambda)^{-1} \otimes  \pi_{13}^*(\mathcal{P}) \otimes \pi_{23}^*(\mathcal{P})  ) ) \\
&=p_*( (\hat{m} \times id)_* ( \pi_{12}^*(\pi_1^*(x) \otimes \pi_2^*(y) ) \otimes \pi_{13}^*(\mathcal{P}) \otimes \pi_{23}^*(\mathcal{P})  ) ) \\
&=p_*( (\hat{m} \times id)_* ( \pi_{13}^*(q^*(x)) \otimes \pi_{23}^*(q^*(y) ) \otimes \pi_{13}^*(\mathcal{P}) \otimes \pi_{23}^*(\mathcal{P})  ) )  \\
&=p_*( (\hat{m} \times id)_* ( \pi_{13}^*(q^*(x) \otimes \mathcal{P}) \otimes \pi_{23}^*(q^*(y) \otimes \mathcal{P} ))) \\
&=p_*( (\pi_{13})_* ( \pi_{13}^*(q^*(x) \otimes \mathcal{P}) \otimes \pi_{23}^*(q^*(y) \otimes \mathcal{P} ))) \\
&=p_*( (q^*(x) \otimes \mathcal{P}) \otimes (\pi_{13})_*(\pi_{23}^*(q^*(y) \otimes \mathcal{P} ))) \\
&=p_*((q^*(x) \otimes \mathcal{P}) \otimes p^* p_*(q^*(y) \otimes \mathcal{P} )) \\
&=p_*(q^*(x) \otimes \mathcal{P}) \otimes p_*(q^*(y) \otimes \mathcal{P} ) \\
&= T(x) \otimes T(y).
\end{aligned}
\end{equation*}

\end{proof}

Since the Fourier-Mukai transform is an isomorphism, this shows that $K^{*-n}(G/T \times \hat{T} , \kappa)$ equipped with the convolution product is a ring isomorphic to  $K^*(G)$.\\

It can be shown that the convolution $*$ induces a multiplicative structure on the Atiyah-Hirzebruch spectral sequence for the fibration $G/T \times \hat{T} \to \hat{T}$ (for instance, one can use the Chern character in twisted $K$-theory to pass to twisted cohomology, where it is easier to describe convolution). This means that $*$ is compatible with the filtration on $K^*(G/T \times \hat{T} , \kappa)$ in the sense that $F^{p,k} * F^{p',k'} \subseteq F^{p+p'-n,k+k'}$. It follows that there is an induced product on the associated graded group $Gr^p( K^q(G/T \times \hat{T} , \kappa)) \simeq E_\infty^{p,q} \simeq E_2^{p,q}$. This is a map of the form $E_2^{p,q} \otimes E_2^{p',q'} \to E_2^{p+p'-n , q+q'}$. Since $E_2^{p,q} = 0$ for odd $q$, we are only concerned with the products $E_2^{p,0} \otimes E_2^{p',0} \to E_2^{p+p'-n,0}$.\\

Let $\rho \colon \pi_1(\hat{T}) \to Aut( K^*(G/T) )$ denote the monodromy representation of the local system $\mathcal{K}(G/T)$. For a space $Z$ and a representation $\phi \colon \pi_1(Z) \to Aut( K(G/T) )$, we write $K(G/T)_\phi$ for the corresponding local system on $Z$. Thus $\mathcal{K}(G/T) = K(G/T)_\rho$. Recall that $E_2^{p,0} = H^p( \hat{T} , K(G/T)_\rho )$. The product $E_2^{p,0} \otimes E_2^{p',0} \to E_2^{p+p'-n , 0}$ is then given by the following composition:
\begin{equation*}
\xymatrix{
H^p( \hat{T} , K(G/T)_\rho ) \otimes H^{p'}(\hat{T} , K(G/T)_\rho ) \ar[d]^{p_1^*( \, \cdot \, ) \smallsmile p_2^*( \, \cdot \, )} \\
H^{p+p'}( \hat{T} \times \hat{T} , K(G/T)_{p_1^*(\rho)} \otimes K(G/T)_{p_2^*(\rho)}) \ar[d]^{\times} \\
H^{p+p'}( \hat{T} \times \hat{T} , K(G/T)_{p_1^*(\rho) + p_2^*(\rho)} ) \ar[d]^{\simeq} \\
H^{p+p'}(\hat{T} \times \hat{T} , K(G/T)_{\hat{\mu}^*(\rho)} ) \ar[d]^{\hat{\mu}_*} \\
H^{p+p'-n}( \hat{T} , K(G/T)_\rho ),
}
\end{equation*}
where $p_1,p_2 \colon \hat{T} \times \hat{T} \to \hat{T}$ are the projections to the first and second factors and $\times \colon K(G/T)_{p_1^*(\rho)} \otimes K(G/T)_{p_2^*(\rho)} \to K(G/T)_{p_1^*(\rho) + p_2^*(\rho)}$ is the homomorphism of local systems given by the product on $K(G/T)$.\\

The convolution is easier to express by switching to Tor groups. Under Poincar\'e duality $Ext^p_{R[T]}(\mathbb{Z} , K(G/T) ) \simeq Tor_{n-p}^{R[T]}(\mathbb{Z} , K(G/T))$, so the product has the form $Tor_{p}^{R[T]}(\mathbb{Z} , K(G/T)) \otimes Tor_{p'}^{R[T]}(\mathbb{Z}, K(G/T)) \to Tor_{p+p'}^{R[T]}(\mathbb{Z} , K(G/T))$. To simplify the notation we let $R = R[G]$, $S = R[T]$ and $K = K(G/T)_\rho$ viewed as an $S$-module. We let $\pi_1,\pi_2 \colon S \otimes_{\mathbb{Z}} S \to S$ be given by $\pi_1 = id \otimes \epsilon$, $\pi_2 = \epsilon \otimes id$ respectively. We also use $m \colon S \otimes_{\mathbb{Z}} S \to S$ to denote the ring multiplication in $S$. The ring structure on $K(G/T)$ defines a module homomorphism $\times \colon \pi_1^*(K) \otimes_{\mathbb{Z}} \pi_2^*(K) \to m^*(K)$. In terms of Tor groups the convolution product is given by the following composition:
\begin{equation*}
\xymatrix{
Tor^S_p(\mathbb{Z},K) \otimes_{\mathbb{Z}} Tor^S_{p'}(\mathbb{Z} , K) \ar[d]^{\otimes} \\
Tor^{S \otimes_{\mathbb{Z}} S}_{p+p'}(\mathbb{Z} , \pi_1^*(K) \otimes_{\mathbb{Z}} \pi_2^*(K) ) \ar[d]^{\times} \\
Tor^{S \otimes_{\mathbb{Z}} S}_{p+p'}(\mathbb{Z} , m^*(K) ) \ar[d]^{m_*} \\
Tor^S_{p+p'}(\mathbb{Z} , K). 
}
\end{equation*}
This is exactly the internal product of Tor groups \cite{macl}.\\

Let $\pi_1,\pi_2, m \colon R \otimes_{\mathbb{Z}} R \to R$ be defined as for $S$. Consider the following diagram:
\begin{equation}\label{diagram}
\xymatrix{
Tor^R_p(\mathbb{Z},\mathbb{Z}) \otimes_{\mathbb{Z}} Tor^R_{p'}(\mathbb{Z} , \mathbb{Z}) \ar[d]^{\otimes} \ar[r] & Tor^S_p(\mathbb{Z},K) \otimes_{\mathbb{Z}} Tor^S_{p'}(\mathbb{Z} , K) \ar[d]^{\otimes} \\
Tor^{R \otimes_{\mathbb{Z}} R}_{p+p'}(\mathbb{Z} , \mathbb{Z} ) \ar[dd]^{m_*} \ar[r] & Tor^{S \otimes_{\mathbb{Z}} S}_{p+p'}(\mathbb{Z} , \pi_1^*(K) \otimes_{\mathbb{Z}} \pi_2^*(K) ) \ar[d]^{\times} \\
 & Tor^{S \otimes_{\mathbb{Z}} S}_{p+p'}(\mathbb{Z} , m^*(K) ) \ar[d]^{m_*} \\
Tor^R_{p+p'}(\mathbb{Z} , \mathbb{Z}) \ar[r] & Tor^S_{p+p'}(\mathbb{Z} , K)
}
\end{equation}
where the horizontal arrows are the natural maps induced by the change of ring spectral sequence for Tor groups.
\begin{proposition}
The diagram (\ref{diagram}) is commutative.
\end{proposition}
\begin{proof}
It is clear that the upper square in (\ref{diagram}) commutes. What needs to be shown is that the lower square also commutes. For this we consider the commutative diagram of rings:
\begin{equation*}
\xymatrix{
R \otimes_{\mathbb{Z}} R \ar[r]^{i \otimes i} \ar[d]^{m} & S \otimes_{\mathbb{Z}} S \ar[d]^{m} \\
R \ar[r]^{i} & S
}
\end{equation*}
where $i \colon R \to S$ is the natural inclusion. This commutative diagram induces a map between the change of ring spectral sequences associated to $(i \otimes i) \colon R \otimes_{\mathbb{Z}} R \to S \otimes_{\mathbb{Z}} S$ and $i \colon R \to S$. Thus we get a commutative square:
\begin{equation*}
\xymatrix{
Tor^{R \otimes_{\mathbb{Z}} R}_{p+p'}(\mathbb{Z} , \mathbb{Z} ) \ar[r] \ar[d]^{m_*} & Tor^{S \otimes_{\mathbb{Z}} S}_{p+p'}(\mathbb{Z} ,  (S \otimes_{\mathbb{Z}} S) \otimes_{R \otimes_{\mathbb{Z}} R} \mathbb{Z} ) \ar[d]^{m_*} \\
Tor^R_{p+p'}(\mathbb{Z} , \mathbb{Z}) \ar[r] & Tor^S_{p+p'}(\mathbb{Z} , S \otimes_{R} \mathbb{Z} )
}
\end{equation*}
We claim that this square coincides with the lower square of (\ref{diagram}). To see this, write $\pi_1^*(K) \otimes_{\mathbb{Z}} \pi_2^*(K)$ as $(S \otimes_{\mathbb{Z}} S) \otimes_{R \otimes_{\mathbb{Z}} R} \mathbb{Z}$. Then the map $\times \colon \pi_1^*(K) \otimes_{\mathbb{Z}} \pi_2^*(K) \to m^*(K)$ is given by:
\begin{equation*}
m \otimes id \colon (S \otimes_{\mathbb{Z}} S) \otimes_{R \otimes_{\mathbb{Z}} R} \mathbb{Z} \to S \otimes_{R} \mathbb{Z}.
\end{equation*}
Making these identifications, it follows that the two squares coincide as claimed.
\end{proof}

We have established that the convolution product on $K^*(G/T \times \hat{T} , \kappa )$ coincides with the product on $Tor^R_*(\mathbb{Z},\mathbb{Z})$ given by the left column of (\ref{diagram}). This is the internal product of Tor groups. Next, we determine the ring structure of $Tor^R_*(\mathbb{Z},\mathbb{Z})$.
\begin{proposition}\label{propextalg}
As a graded ring $Tor^R_*(\mathbb{Z},\mathbb{Z})$ is isomorphic to an exterior algebra $\bigwedge^*_{\mathbb{Z}} \{ \rho_1 , \dots , \rho_n \}$ over $\mathbb{Z}$ on $n$ generators, where the $\rho_i$ have degree $1$. 
\end{proposition}
\begin{proof}
This follows easily by taking the tensor product of two Koszul resolutions for $\mathbb{Z}$ over $R$ \cite{macl}.
\end{proof}

\begin{corollary}
The twisted $K$-theory $K^{*-n}(G/T \times \hat{T} , \kappa)$ with convolution product is isomorphic to the exterior algebra  $\bigwedge^*_{\mathbb{Z}} \{ \rho_1 , \dots , \rho_n \}$.
\end{corollary}
\begin{proof}
We have that $K^{*-n}(G/T \times \hat{T} , \kappa)$ admits a filtration for which the associated graded ring, by Proposition \ref{propextalg}, is an exterior algebra $\bigwedge^*_{\mathbb{Z}} \{ \rho_1 , \dots , \rho_n \}$. Since $\bigwedge^1_{\mathbb{Z}}\{ \rho_1 , \dots , \rho_n \} \simeq F^{n-1,n-1} \subseteq K^{n-1}(G/T \times \hat{T} , \kappa)$, there are canonical lifts $\tilde{\rho}_1, \dots , \tilde{\rho}_n \in K^{n-1}(G/T \times \hat{T} , \kappa)$ of $\rho_1, \dots , \rho_n$. Note that $K^{*-n}(G/T \times \hat{T} , \kappa)$ is a ring with identity as it is isomorphic to $K^*(G)$. For reasons of degree, the identity must correspond to a generator of $\mathbb{Z} \simeq F^{n,n} \subseteq K^n(G/T \times \hat{T} , \kappa)$. Comparing with the associated graded ring, it is clear that $\tilde{\rho}_1, \dots , \tilde{\rho}_n$ together with the identity generate the whole of $K^{*-n}(G/T \times \hat{T} , \kappa)$. The elements $\tilde{\rho}_i$ anti-commute, since they map to elements of $K^1(G)$ under the Fourier-Mukai transform. Thus $K^{*-n}(G/T \times \hat{T} , \kappa)$ is isomorphic to a quotient of $\bigwedge^*_{\mathbb{Z}} \{ \tilde{\rho}_1 , \dots , \tilde{\rho}_n \}$. Any non-trivial quotient will have rank less than $2^n$, hence we must have $K^{*-n}(G/T \times \hat{T} , \kappa) \simeq \bigwedge^*_{\mathbb{Z}} \{ \rho_1 , \dots , \rho_n \}$.
\end{proof}

This concludes our proof of Theorem \ref{thmhod}.


\bibliographystyle{amsplain}

\end{document}